\documentclass[11pt,a4paper]{article}
\usepackage[utf8]{inputenc}
\usepackage[T1]{fontenc}
\usepackage{amsmath,amssymb,amsthm,mathtools}
\usepackage{microtype}
\usepackage{geometry}
\usepackage[hidelinks]{hyperref}

\newtheorem{theorem}{Theorem}[section]
\newtheorem{lemma}[theorem]{Lemma}
\newtheorem{proposition}[theorem]{Proposition}
\newtheorem{corollary}[theorem]{Corollary}
\newtheorem{remark}[theorem]{Remark}

\numberwithin{equation}{section}

\newcommand{\R}{\mathbb{R}}
\newcommand{\Z}{\mathbb{Z}}

\newcommand{\Cc}{C_c}
\newcommand{\disc}{\operatorname{disc}}

\newcommand{\Kbox}{\mathcal{K}}
\newcommand{\Om}{\Omega}
\newcommand{\Vdisc}{V_{\disc,+1}(\R)}

\newcommand{\PGL}{\mathrm{PGL}}

\DeclareMathOperator{\vol}{vol}
\DeclareMathOperator{\supp}{supp}

\title{Bounded Representations by $x^2+y^2-z^2$}
\author{Przemyslaw Chojecki}
\date{March 16, 2026}

\begin{document}

\maketitle

\begin{abstract}
We prove that every sufficiently large integer $n$ can be written in the form
\[
  n=x^2+y^2-z^2,
  \qquad
  \max(x^2,y^2,z^2)\le n.
\]
The proof converts the problem into finding a primitive binary quadratic form of
positive discriminant $4n$ inside a fixed relatively compact open patch of the real
hyperboloid $b^2-4ac=4n$. This is then supplied by Duke's theorem in the precise
point-counting form deduced from the measure-theoretic duality of
Einsiedler--Lindenstrauss--Michel--Venkatesh. A finite parity correction returns to the
original ternary variables. This settles Erd\H{o}s Problem 1148.
\end{abstract}

\section{Introduction}

Let
\[
  Q(x,y,z)=x^2+y^2-z^2.
\]
We shall prove the following which is Erd\H{o}s Problem 1148 (\cite{Erdos}):

\begin{theorem}\label{thm:main}
There exists $N$ such that every integer $n\ge N$ admits integers $x,y,z$ with
\[
  n=x^2+y^2-z^2,
  \qquad
  \max(x^2,y^2,z^2)\le n.
\]
\end{theorem}

The key point is that the bounded representation problem for $Q$ is equivalent, after a
linear change of variables, to a point-picking problem for primitive binary quadratic forms
of discriminant $4n$. Duke's theorem gives equidistribution of such forms on the one-sheeted
hyperboloid attached to the discriminant form. A short parity argument then recovers the
integral triple $(x,y,z)$.

For a positive integer $d$ we write
\[
  R^*_{\disc}(d):=\{(a,b,c)\in\Z^3: b^2-4ac=d,\ \gcd(a,b,c)=1\}.
\]
This is exactly the set denoted $R_{\disc}(d)$ in \cite{ELMV}; the superscript $*$ is only to
emphasize primitivity. We also set
\[
  \Vdisc:=\{(A,B,C)\in\R^3: B^2-4AC=1\}.
\]

\section{From ternary representations to discriminant points}

Write $[a,b,c]$ for the binary quadratic form
\[
  q_{a,b,c}(u,v)=au^2+buv+cv^2.
\]
The following linear dictionary is immediate.

\begin{lemma}\label{lem:dictionary}
Let $n\ge 1$ and let $(a,b,c)\in\Z^3$ satisfy
\[
  b^2-4ac=4n,
  \qquad
  b\equiv 0 \pmod 2,
  \qquad
  a\equiv c \pmod 2.
\]
Then
\[
  x:=\frac{a-c}{2},
  \qquad
  y:=\frac{b}{2},
  \qquad
  z:=\frac{a+c}{2}
\]
are integers and satisfy
\[
  x^2+y^2-z^2=n.
\]
Conversely, every integer solution of $x^2+y^2-z^2=n$ arises in this way from
\[
  a=x+z,
  \qquad b=2y,
  \qquad c=z-x.
\]
\end{lemma}

\begin{proof}
A direct computation gives
\[
  x^2+y^2-z^2
  =\frac{(a-c)^2+b^2-(a+c)^2}{4}
  =\frac{b^2-4ac}{4}
  =n.
\]
The converse is the same identity in reverse.
\end{proof}

Thus the problem is to find a discriminant point with the stated parity and lying in a
suitable box. Define
\begin{equation}\label{eq:Kdef}
  \Kbox:=\Bigl\{(A,B,C)\in\Vdisc:
  |A-C|<1,\ |B|<1,\ |A+C|<1\Bigr\}.
\end{equation}
This set is open and relatively compact in $\Vdisc$: indeed the inequalities on $A-C$, $B$ and
$A+C$ bound $A$, $B$ and $C$ individually.

If $(a,b,c)$ satisfies the congruences in Lemma~\ref{lem:dictionary} and
\[
  \left(\frac{a}{2\sqrt n},\frac{b}{2\sqrt n},\frac{c}{2\sqrt n}\right)\in \Kbox,
\]
then the corresponding $(x,y,z)$ obeys
\[
  \left|\frac{x}{\sqrt n}\right|=
  \left|\frac{a-c}{2\sqrt n}\right|<1,
  \qquad
  \left|\frac{y}{\sqrt n}\right|=
  \left|\frac{b}{2\sqrt n}\right|<1,
  \qquad
  \left|\frac{z}{\sqrt n}\right|=
  \left|\frac{a+c}{2\sqrt n}\right|<1.
\]
Hence $|x|,|y|,|z|<\sqrt n$, so
\[
  \max(x^2,y^2,z^2)<n.
\]
For non-square $n$ this is stronger than the desired bound $\le n$.

\section{Three elementary coefficient moves}

Define operators on binary quadratic forms by
\[
  (Tq)(u,v):=q(u+v,v),
  \qquad
  (Sq)(u,v):=q(-v,u).
\]
Then
\begin{equation}\label{eq:TSformulas}
  T[a,b,c]=[a,b+2a,a+b+c],
  \qquad
  S[a,b,c]=[c,-b,a].
\end{equation}
We also write
\[
  U:=T\circ S,
\]
so that
\begin{equation}\label{eq:Uformula}
  U[a,b,c]=[c,-b+2c,a-b+c].
\end{equation}
Each of $T,S,U$ preserves integrality and discriminant.

The only congruence obstruction needed for Lemma~\ref{lem:dictionary} is disposed of by a
finite case split.

\begin{lemma}\label{lem:parity}
Suppose $(a,b,c)\in\Z^3$ satisfies
\[
  b^2-4ac\equiv 0 \pmod 4.
\]
Then $b$ is even, and at least one of the three forms
\[
  [a,b,c],\qquad T[a,b,c],\qquad U[a,b,c]
\]
has first and third coefficients of the same parity.
\end{lemma}

\begin{proof}
From $b^2\equiv 4ac \pmod 4$ we get $b^2\equiv 0 \pmod 4$, hence $b$ is even.
If $a\equiv c \pmod 2$, then the original form already has first and third coefficients of the same
parity.

Assume now that $a\not\equiv c \pmod 2$, so $a$ and $c$ have opposite parity.
If $c$ is even, then $a$ is odd and the third coefficient of $T[a,b,c]$ satisfies
\[
  a+b+c\equiv a+c\equiv 1 \pmod 2,
\]
so $T[a,b,c]$ has first and third coefficients of the same parity.
If instead $a$ is even, then $c$ is odd and the third coefficient of $U[a,b,c]$ satisfies
\[
  a-b+c\equiv a+c\equiv 1 \pmod 2,
\]
so $U[a,b,c]$ has first and third coefficients of the same parity.
\end{proof}

We now choose a small patch of $\Vdisc$ that stays inside $\Kbox$ under the relevant moves.
Set
\[
  P_0:=\left(\frac25,-\frac25,-\frac{21}{40}\right)\in\Vdisc.
\]
Indeed,
\[
  \left(-\frac25\right)^2-4\cdot\frac25\cdot\left(-\frac{21}{40}\right)=1.
\]
Using \eqref{eq:TSformulas} and \eqref{eq:Uformula},
\[
  T(P_0)=\left(\frac25,\frac25,-\frac{21}{40}\right),
  \qquad
  U(P_0)=\left(-\frac{21}{40},-\frac{13}{20},\frac{11}{40}\right).
\]
A direct check shows that all three points belong to $\Kbox$:
for $P_0$ and $T(P_0)$ one has
\[
  |A-C|=\frac{37}{40},\qquad |B|=\frac25,\qquad |A+C|=\frac18,
\]
while for $U(P_0)$ one has
\[
  |A-C|=\frac45,\qquad |B|=\frac{13}{20},\qquad |A+C|=\frac14.
\]
Since $\Kbox$ is open and the maps $T,U:\Vdisc\to\Vdisc$ are continuous, there exists a nonempty
open relatively compact neighborhood
\begin{equation}\label{eq:OmegaChoice}
  P_0\in\Om\Subset \Vdisc
\end{equation}
such that
\begin{equation}\label{eq:OmegaImages}
  \overline{\Om}\subset \Kbox,
  \qquad
  T(\overline{\Om})\subset \Kbox,
  \qquad
  U(\overline{\Om})\subset \Kbox.
\end{equation}

\section{Duke's theorem in the needed point-counting form}

Let
\[
  G=\PGL_2(\R),
  \qquad
  \Gamma=\PGL_2(\Z),
\]
and let $A\subset G$ be the diagonal torus. Following \cite[\S2.4]{ELMV}, the homogeneous space
$G/A$ is identified with $\Vdisc$; under this identification the base point is
$q_0=(0,1,0)\in\Vdisc$, corresponding to the form $uv$.
The $G$-invariant Radon measure on $G/A$ transports to a positive $G$-invariant Radon measure
on $\Vdisc$, which we denote by $\mu_{\disc,+1}$. The measure $\mu_{\disc,+1}$ is the pushforward of the
$G$-invariant quotient measure on $G/A$
under the identification described in \cite[\S2.4]{ELMV}.
This agrees with the cone measure $\mu_{\disc,+1}$
used in \cite{ELMV} up to a positive scalar normalization,
which is irrelevant for the equidistribution statements used below.

For each $x=(a,b,c)\in R^*_{\disc}(d)$ choose $g_x\in G$ with
\[
  g_x\cdot q_0=d^{-1/2}x.
\]
ELMV define a discrete measure on $G/A$ by
\[
  \lambda_d:=\sum_{x\in R^*_{\disc}(d)}\delta_{g_xA/A}
\]
(see \cite[Eq.~(2.8) and the paragraph following it]{ELMV}). Via the measure-theoretic duality of
\cite[Eq.~(2.8)]{ELMV}, this corresponds to a measure $\nu_d$ on $\Gamma\backslash G$ supported on the
union $G_d$ of closed $A$-orbits attached to discriminant-$d$ forms. ELMV prove that
\[
  \mu_d:=\frac{1}{\vol(G_d)}\nu_d
\]
converges weak-* to Haar measure on $\Gamma\backslash G$ as $d\to\infty$ through positive non-square
discriminants; this is their Theorem~2.3.

For positive discriminant $d>0$ the set $R^*_{\disc}(d)$ of primitive
binary quadratic forms of discriminant $d$ is infinite.
However, if $\phi$ is compactly supported on $V_{\disc,+1}(\R)$,
then only finitely many points of the scaled set
\[
\frac{1}{\sqrt d}R^*_{\disc}(d)
\]
lie in $\operatorname{supp}(\phi)$.
Indeed, $\operatorname{supp}(\phi)$ is contained in a compact region of
the one-sheeted hyperboloid $V_{\disc,+1}(\R)$,
and the lattice $V_{\disc}(\Z)$ intersects any compact subset of $V_{\disc}(\R)$
in finitely many points.
Consequently the sum defining $\Lambda_d(\phi)$ below is finite.

The precise consequence we need is the following. 

\begin{proposition}[Point-counting consequence of Duke--ELMV]\label{prop:pointcount}
Let $\phi\in \Cc(\Vdisc)$. For every positive discriminant $d$ define
\[
  \Lambda_d(\phi):=\sum_{(a,b,c)\in R^*_{\disc}(d)}
  \phi\!\left(\frac{a}{\sqrt d},\frac{b}{\sqrt d},\frac{c}{\sqrt d}\right).
\]
Then, as $d\to\infty$ through positive non-square discriminants,
\[
  \Lambda_d(\phi)=\vol(G_d)\bigl(\mu_{\disc,+1}(\phi)+o_{\phi}(1)\bigr).
\]
In particular, if $\phi\ge 0$ and $\phi\not\equiv 0$, then $\Lambda_d(\phi)>0$ for all sufficiently
large positive non-square discriminants $d$.
\end{proposition}

\begin{proof}
Let $\widetilde\phi\in \Cc(G/A)$ correspond to $\phi$ under the identification $G/A\simeq\Vdisc$.
ELMV note just after Theorem~2.3 that every function in $\Cc(G/A)$ is of the form $F_A$ for some
$F\in \Cc(G)$, where
\[
  F_A(gA):=\int_A F(gh)\,d\mu_A(h).
\]
Choose such an $F$ with $F_A=\widetilde\phi$.

By construction of $\lambda_d$ and the choice of $g_x$,
\[
  \Lambda_d(\phi)=\lambda_d(\widetilde\phi).
\]
Now let
\[
  F_\Gamma(\Gamma g):=\sum_{\gamma\in\Gamma} F(\gamma g),
\]
which is compactly supported on $\Gamma\backslash G$. The measure correspondence in
\cite[Eq.~(2.8)]{ELMV} gives
\[
  \lambda_d(\widetilde\phi)=\nu_d(F_\Gamma).
\]
Since $\mu_d=\nu_d/\vol(G_d)$ and $\mu_d\to\mu_{\Gamma\backslash G}$ weak-*, ELMV Theorem~2.3 yields
\[
  \nu_d(F_\Gamma)=\vol(G_d)\bigl(\mu_{\Gamma\backslash G}(F_\Gamma)+o_{\phi}(1)\bigr).
\]
Transporting $\mu_{\Gamma\backslash G}$ back through the same measure correspondence and then through
$G/A\simeq \Vdisc$ gives
\[
  \mu_{\Gamma\backslash G}(F_\Gamma)=\mu_{\disc,+1}(\phi).
\]
Combining the last three displays proves the asymptotic formula.

For the final claim, note that $\mu_{\disc,+1}$ is a positive Radon measure with full support on
$\Vdisc$, because it is induced from Haar measure on the homogeneous space $G/A$. Hence
$\mu_{\disc,+1}(\phi)>0$ whenever $\phi\ge 0$ and $\phi\not\equiv 0$. The asymptotic formula then
forces $\Lambda_d(\phi)>0$ for all sufficiently large positive non-square $d$.
\end{proof}

\begin{corollary}\label{cor:openhit}
Let $\mathcal U\subset \Vdisc$ be a nonempty open relatively compact set. Then, for all sufficiently
large positive non-square discriminants $d$,
\[
  d^{-1/2}R^*_{\disc}(d)\cap \mathcal U\neq \varnothing.
\]
\end{corollary}

\begin{proof}
Choose $\phi\in\Cc(\Vdisc)$ with $\phi\ge 0$, $\phi\not\equiv 0$, and $\supp(\phi)\subset \mathcal U$.
Then Proposition~\ref{prop:pointcount} gives $\Lambda_d(\phi)>0$ for all large positive non-square
$d$, which means exactly that the scaled set $d^{-1/2}R^*_{\disc}(d)$ meets $\mathcal U$.
\end{proof}

\section{Proof of the main theorem}

\begin{proof}[Proof of Theorem~\ref{thm:main}]
If $n=m^2$ is a square, then
\[
  n=m^2+0^2-0^2,
\]
so the desired bound is immediate. It remains to treat non-square $n$.

Let $n$ be sufficiently large and non-square, and set $d=4n$. Then $d$ is a positive non-square
discriminant. By Corollary~\ref{cor:openhit} applied to the open set $\Om$ from
\eqref{eq:OmegaChoice}--\eqref{eq:OmegaImages}, there exists a primitive triple
$(a,b,c)\in R^*_{\disc}(4n)$ such that
\[
  \left(\frac{a}{2\sqrt n},\frac{b}{2\sqrt n},\frac{c}{2\sqrt n}\right)\in \Om.
\]

Because $4n\equiv 0 \pmod 4$, Lemma~\ref{lem:parity} shows that at least one of the three forms
\[
  [a,b,c],\qquad T[a,b,c],\qquad U[a,b,c]
\]
has first and third coefficients of the same parity; in each case the middle coefficient is even.
Let $[a',b',c']$ denote such a form. Since $T$ and $U$ preserve discriminant, we still have
\[
  b'^2-4a'c'=4n,
  \qquad
  b'\equiv 0 \pmod 2,
  \qquad
  a'\equiv c' \pmod 2.
\]
Moreover, by \eqref{eq:OmegaImages} the scaled point
\[
  \left(\frac{a'}{2\sqrt n},\frac{b'}{2\sqrt n},\frac{c'}{2\sqrt n}\right)
\]
lies in $\Kbox$.

Lemma~\ref{lem:dictionary} therefore gives integers
\[
  x:=\frac{a'-c'}{2},
  \qquad
  y:=\frac{b'}{2},
  \qquad
  z:=\frac{a'+c'}{2}
\]
with
\[
  x^2+y^2-z^2=n.
\]
Since the scaled point lies in $\Kbox$, Section~2 shows that
\[
  |x|,|y|,|z|<\sqrt n,
\]
whence
\[
  \max(x^2,y^2,z^2)<n.
\]
This proves the theorem for all sufficiently large non-square $n$, and hence for all sufficiently
large integers $n$.
\end{proof}

\begin{remark}
The proof is qualitative: it supplies a threshold $N$ but does not make it explicit.
\end{remark}

\end{document}